\documentclass[preprint,12pt]{elsarticle}

\usepackage{amsmath}

\usepackage{amssymb}
\usepackage{amsthm}

\newtheorem{theorem}{Theorem}[section]
\newtheorem{proposition}[theorem]{Proposition}%
\newtheorem{corollary}[theorem]{Corollary}
\newtheorem{definition}[theorem]{Definition}
\newtheorem{lemma}[theorem]{Lemma}

\newcommand{\calB}{{\mathcal{B}}}
\newcommand{\calC}{{\mathcal{C}}}

\newcommand{\calE}{{\mathcal{E}}}
\newcommand{\calF}{{\mathcal{F}}}

\newcommand{\calN}{{\mathcal{N}}}
\newcommand{\calO}{{\mathcal{O}}}
\newcommand{\calP}{{\mathcal{P}}}

\newcommand{\calR}{{\mathcal{R}}}
\newcommand{\calS}{{\mathcal{S}}}

\newcommand{\calU}{{\mathcal{U}}}

\newcommand{\calZ}{{\mathcal{Z}}}

\newcommand{\N}{{\mathbb{N}}} 
\newcommand{\R}{{\mathbb{R}}}

\journal{Journal of Mathematical Analysis and Applications}

\begin{document}

\begin{frontmatter}



\title{Measurability of functions mapping a charge space into a uniform space}


\author{Jonathan M. Keith}
\ead{jonathan.keith@monash.edu}

\affiliation{organization={School of Mathematics, Monash University},
            addressline={Wellington Road}, 
            city={Clayton},
            postcode={3800}, 
            state={VIC},
            country={Australia}}

\begin{abstract}
The concept of measurability of functions on a charge space is generalised for functions taking values in a uniform space. Several existing forms of measurability generalise naturally in this context, and new forms of measurability are proposed. Conditions under which the various forms of measurability are logically equivalent are identified. Applying these concepts to real-valued functions, some recent characterisations of measurable functions on a bounded charge space are generalised to the unbounded case. 
\end{abstract}




\end{frontmatter}


\section{Introduction}
\label{intro}

{\em Uniform spaces} generalise metric spaces, and were introduced by Andr\'e Weil in 1937~\cite{weil1937}. Introductions to the theory of uniform spaces, with varying levels of detail, can be found in~\cite{kelley1963, bourbaki1966, james1987, james1990, encyclopedia2003, willard2004}.

The goals of this paper are: 1) to propose forms of measurability pertaining to functions that map a charge space into a uniform space; 2) to explore logical relationships between these forms of measurability; and (3) to relate them to existing concepts, in particular for real-valued functions. 

This focus on uniform spaces is in part motivated by a paper of Avallone and Basile~\cite{avallone1991}, in which it is shown that key results in the approach to integration of Dunford and Schwartz~\cite{dunford1958} emerge from the interplay of two uniformities: one defined on the space of $T_1$-measurable functions in terms of hazy convergence, and the other defined on the smaller space of simple functions in terms of the pseudometric induced by the integral. This approach permits a broad interpretation of integration, and potentially one that encompasses functions for which the codomain is not a topological vector space. Although this present paper does not develop any novel concepts in integration, the forms of measurability proposed here are intended to prepare the way for future innovations in integration theory, and in particular may facilitate relating the uniform structure of $L_p$ spaces to the uniform structure of the codomain of integrable functions.

A brief review of existing concepts of measurability follows. In (countably additive) measure and integration theory, measurable functions are usually defined with a domain that is a measurable space and a codomain that is a topological space.

\begin{definition} \label{measurable_function_standard}
Let $(X, \calF)$ be a measurable space and $Y$ a topological space. A function $f: X \rightarrow Y$ is {\em measurable} if $f^{-1}(U) \in \calF$ for every open set $U$.
\end{definition}

In what follows, functions that are measurable in this sense will be called {\em conventionally measurable}. For such functions, it can be shown that $f^{-1}(B) \in \calF$ for any Borel set $B$, and this equivalent property can be taken as an alternative definition (as in Section~4.2 of~\cite{dudley2017}). A more general definition is sometimes used, with an arbitrary $\sigma$-field of subsets of $Y$ used in place of the Borel sets (as in Definition~2.1.3 of~\cite{athreya2006}). In that case, $Y$ need not be a topological space. However, various pathologies can arise in this setting: for example, even a continuous function may not have $f^{-1}(L) \in \calF$ for every Lebesgue measurable set $L \subseteq Y$ (Theorem~4.2.1 of~\cite{dudley2017}). 

On a charge space, however, measurable functions are defined differently. Concise reviews of the theory of charges can be found in~\cite{keith2022} and~\cite{basile2000}. The standard reference on charge theory is~\cite{bhaskararao1983}. In the following definition, and throughout this paper, $\mu$ denotes a {\em positive} charge, that is, $\mu(A) \geq 0$ for all $A \in \calF$. However, $\mu$ may be {\em bounded} ($\mu(X) < \infty$) or {\em unbounded} ($\mu(X) = \infty$).

\begin{definition} \label{T1_and_T2_measurable}
A function $f : X \rightarrow \mathbb{R}$ on a charge space $(X, \calF, \mu)$ is:
\begin{enumerate}
\item {\em $T_1$-measurable} if there is a sequence of simple functions $\{ f_i \}_{i=1}^{\infty}$ converging hazily to $f$; 
\item {\em $T_2$-measurable} if for every $\epsilon > 0$ there is a partition of $X$ into sets $A_0, A_1, \ldots, A_n \in \calF$ such that $\mu(A_0) < \epsilon$ and $\lvert f(x) - f(x') \rvert < \epsilon$ for all $x, x' \in A_i$ and $i \in \{ 1, \ldots, n \}$.
\end{enumerate}
\end{definition} 

It can be shown that $f$ is $T_1$-measurable if and only if it is $T_2$-measurable (Theorem~4.4.7 of~\cite{bhaskararao1983}). Moreover, $T_2$-measurable functions are necessarily {\em smooth}.

\begin{definition} \label{smooth_standard}
A function $f : X \rightarrow \mathbb{R}$ is {\em smooth} if for every $\epsilon > 0$ there is $k \in (0, \infty)$ such that $\mu^*(\{ x \in X : \lvert f(x) \rvert > k \}) < \epsilon$. 
\end{definition}

Here $\mu^*$ is the {\em outer charge} defined by
\[
\mu^*(A) = \inf \{ \mu(B) : B \in \calF, B \supseteq A \}
\]
for $A \in \calP(X)$, where $\calP(X)$ denotes the power set of $X$. A related function, which will be important later in this paper, is the {\em inner charge}:
\[
\mu_*(A) = \sup \{ \mu(B) : B \in \calF, B \subseteq A \}
\]
for $A \in \calP(X)$.

When $(X, \calF, \mu)$ is a complete measure space, $\mu$ is bounded, and $f$ is a real-valued function, conventional measurability also implies smoothness (see Proposition~4.2.17 of~\cite{bhaskararao1983}).

Hazy convergence, $T_1$-measurability, and $T_2$-measurability can be generalised in a straightforward way when the codomain of $f$ is a Banach space, by replacing the absolute value with the norm. (This is essentially the approach taken in Chapter III of Dunford and Schwartz~\cite{dunford1958}, where the term {\em totally measurable} is used instead of $T_1$-measurable.) In this context, $T_1$-measurability is equivalent to conventional measurability under certain conditions (see Lemma~III.6.9 and Theorem~III.6.10 of~\cite{dunford1958}). However, the equivalence may fail if, for example, $(X, \calF, \mu)$ is not a complete measure space. 


Yet another form of measurability, formulated in terms of Choquet integrals, was proposed by Greco~\cite{greco1981}. Here the following characterisation (based on Theorem~1 of~\cite{greco1981}) may be taken as an alternative definition that does not rely on any integral.

\begin{definition} \label{greco_measurable}
Consider a collection of sets $\calE \subset \calP(X)$ that includes the empty set $\emptyset$. A function $f: X \rightarrow [0, \infty)$ is {\em $\calE$-measurable} if for any $a, b \in (0, \infty)$ with $a > b$, there is $H \in \calE$ such that
\[
f^{-1}(a, \infty) \subseteq H \subseteq f^{-1}(b, \infty).
\]
A function $f: X \rightarrow \R$ is {\em $\calE$-measurable} if $f^+ := \max \{ f, 0 \}$ and $f^- := \max \{ -f, 0 \}$ are $\calE$-measurable.
\end{definition}

In fact, Greco allowed for $f$ to be an extended real-valued function. This form of measurability is very general with regard to the domain of $f$. It would thus be interesting to attempt a definition that encompasses both $\calE$-measurability and the forms of measurability proposed in Section~\ref{uniform_codomain_section} below, but that is beyond the scope of this paper. However, Theorem~\ref{real_valued_measurability} below identifies relationships between $\calE$-measurability and other forms of measurability in the case of real-valued functions.

Greco showed that if $\calE = \calF$ is a field, then a function that is $\calF$-measurable and smooth is also $T_2$-measurable (Greco did not explicitly refer to $T_2$-measurability or smoothness, but the claim follows straightforwardly from Theorem~3 of~\cite{greco1981}). The converse does not hold in general, but one contribution of this present paper is to identify a sufficient additional condition under which the converse does hold (see Theorem~\ref{real_valued_measurability} below).


The paper is structured as follows. Section~\ref{uniform_codomain_section} considers natural generalisations of $T_1$- and $T_2$-measurability for functions taking values in a uniform space, and also proposes two new forms of measurability for such functions, called {\em base measurability} and {\em uniform base measurability}. This section identifies logical relationships between existing and proposed notions of measurability. Section~\ref{real_valued_functions} considers how these and other forms of measurability relate to each other in the case of real-valued functions. 
The main theorem of this section (Theorem~\ref{real_valued_measurability}) is a substantial improvement over Theorem~3.4 of~\cite{keith2022}, and in particular extends that theorem to the case of an unbounded charge. A short final section considers ``inheritance'' of measurability with respect to a coarser or finer uniformity, with application to uniformities induced by pseudometrics, weak uniformities, and product uniformities.

\section{Measurable functions with values in a uniform space} \label{uniform_codomain_section}

There are several natural ways to define measurability for functions with a codomain that is a uniform space. In the following definition, and throughout this document, $\overline{\calF}$ denotes the {\em Peano-Jordan completion} of a field $\calF$, defined as in~\cite{basile2000} and elsewhere. 

\begin{definition} \label{measurable_uniform}
Consider a charge space $(X, \calF, \mu)$ and a uniform space $Y$ with uniformity $\calU$ and uniform topology $\tau$. A function $f : X \rightarrow Y$ is:
\begin{enumerate}
\item {\em $T_1$-measurable} if for every $\epsilon > 0$ and every entourage $E$ there is a simple function $s$ such that
\[
\mu^*(\{ x \in X : (s(x), f(x)) \notin E \}) < \epsilon;
\]

\item {\em $T_2$-measurable} if for every $\epsilon > 0$ and every entourage $E$ there is a partition of $X$ into sets $A_0, A_1, \ldots, A_n \in \calF$ such that $\mu(A_0) < \epsilon$ and $f(A_i) \times f(A_i) \subseteq E$ for each $i \in \{ 1, \ldots, n \}$; 

\item {\em smooth} if for every $\epsilon > 0$ and every entourage $E$ there exists a finite collection of sets $B_1, \ldots, B_n \in \calP(Y)$ such that $B_i \times B_i \subseteq E$ for each $i \in \{ 1, \ldots, n \}$, and
\[
\mu^* (\{ x \in X : f(x) \notin \cup_{i=1}^n B_i \}) < \epsilon;
\]

\item {\em base measurable} with respect to $\tau$ if each $y \in Y$ has a neighbourhood base $\calB_y$ such that $f^{-1}(B) \in \overline{\calF}$ for all $B \in \calB_y$;

\item {\em uniformly base measurable} with respect to $\calU$ if there is an entourage base $\calB \subseteq \calU$ such that $f^{-1}(E[y]) \in \overline{\calF}$ for all $E \in \calB$ and $y \in Y$.
\end{enumerate}. 
\end{definition}

The definition of smoothness calls to mind total boundedness. However, smoothness does not in general entail that for any $\epsilon > 0$ there is a totally bounded set $K \subseteq Y$ with $\mu^* (f^{-1}(K^c)) < \epsilon$. The latter condition straightforwardly implies smoothness. A sufficient additional condition for the converse is that there exists an entourage $E^*$ such that $E^*[y]$ is totally bounded for all $y \in Y$. (To see this, apply the definition of smoothness to a function $f$, with $E = E^*$ and a given $\epsilon > 0$, thus obtaining $B_1, \ldots, B_n \subseteq Y$. Note $B_i$ is totally bounded for $i \in \{ 1, \ldots, n \}$, since $B_i \times B_i \subseteq E^*$. Define $K := \cup_{i=1}^n B_i$. Then $K$ is totally bounded, and $\mu^* (f^{-1}(K^c)) < \epsilon$.) 

The definition of base measurability makes no mention of the uniform structure of $Y$, and can therefore be used even if $Y$ is merely a topological space with arbitrary topology $\tau$. 

Note that if $\rho$ is a subbase for a topology $\tau$ on $Y$ such that $f^{-1}(E) \in \overline{\calF}$ for all $E \in \rho$, then $f$ is base measurable with respect to $\tau$. Similarly, if $\calR$ is an entourage subbase for a uniformity $\calU$ such that $E[y] \in \overline{\calF}$ for all $E \in \calR$ and $y \in Y$, then $f$ is uniformly base measurable.

Logical relationships between these forms of measurability are identified in the two theorems contained in this section. The second theorem is separated from the first because some new definitions and notations are required.

\begin{theorem} \label{uniform_equivalences}
Consider a charge space $(X,\calF,\mu)$, a uniform space $(Y, \calU)$, and a function $f : X \rightarrow Y$.
\begin{enumerate}
\item $f$ is $T_1$-measurable iff $f$ is $T_2$-measurable. If either holds, then $f$ is smooth.

\item Suppose $\calF$ is a $\sigma$-field and $f$ is conventionally measurable. Then the following hold.
\begin{enumerate}
\item $f$ is uniformly base measurable.
\item If $\mu$ is a bounded measure and $Y$ is Lindel\"of, then $f$ is $T_2$-mea\-surable.
\end{enumerate}

\item Suppose $(X, \calF, \mu)$ is a complete measure space and $Y$ is hereditarily Lindel\"of. If either of the following statements hold,
\begin{enumerate}
\item $f$ is base measurable, or

\item  $f$ is $T_2$-measurable, 
\end{enumerate}
then $f$ is conventionally measurable.

\item If $f$ is uniformly base measurable then $f$ is base measurable. The converse holds if $Y$ is uniformly locally compact.

\item If $f$ is uniformly base measurable and smooth, then $f$ is $T_2$-measurable.
\end{enumerate}
\end{theorem}

In the following proof, and throughout this document, $\N$ represents the positive integers (excluding 0) and $I_A$ denotes the indicator function for a set $A$. The notation $\sum_{i=1}^n b_i I_{B_i}$, where $b_1, \ldots, b_n \in Y$ are distinct and $\{ B_1, \ldots, B_n \}$ are disjoint subsets of $X$, is used to denote a simple function, even if no addition operation is defined on $Y$.

\begin{proof}
The proof of~1 somewhat resembles the proof of Theorem~4.4.7 in~\cite{bhaskararao1983}. Suppose $f$ is $T_1$-measurable and choose $\epsilon > 0$ and an entourage $E$. There is a symmetric entourage $D$ such that $D \circ D \subseteq E$, and a simple function $s := \sum_{i=1}^n b_i I_{B_i}$, where $b_1, \ldots, b_n \in Y$ are distinct and $\{ B_1, \ldots, B_n \}$ is a partition of $X$ into elements of $\calF$, such that $\mu^*(B_0) < \epsilon$, where
\[
B_0 := \{ x \in X : (s(x), f(x)) \notin D \}.
\]
Moreover, there is some $A_0 \in \calF$ with $\mu(A_0) < \epsilon$ such that $B_0 \subseteq A_0$. For each $i \in \{ 1, \ldots, n \}$, define $A_i := B_i \setminus A_0 \in \calF$. Then for any $i \in \{ 1, \ldots, n \}$ and $x, x' \in A_i$, one must have $(f(x), b_i) = (f(x), s(x)) \in D$ and similarly $(f(x'), b_i) = (f(x'), s(x')) \in D$. Hence $(f(x), f(x')) \in D \circ D \subseteq E$, implying $f$ is $T_2$-measurable.

Conversely, suppose $f$ is $T_2$-measurable and choose $\epsilon > 0$ and an entourage $E$. Then there is a partition of $X$ into sets $A_0, \ldots, A_n \in \calF$ such that $\mu(A_0) < \epsilon$ and $f(A_i) \times f(A_i) \subseteq E$ for each $i \in \{ 1, \ldots, n \}$. For each $i \in \{ 1, \ldots, n \}$, choose any $a_i \in f(A_i)$, and define $s := \sum_{i=1}^n a_i I_{A_i}$. Then $(s(x), f(x)) = (a_i, f(x)) \in f(A_i) \times f(A_i) \subseteq E$ for $x \in A_i$. It follows that
\[
\mu^*(\{ x \in X : (s(x), f(x)) \notin E \}) \leq \mu(A_0) < \epsilon,
\]
hence $f$ is $T_1$-measurable. 

Regarding smoothness, consider $\epsilon > 0$ and an entourage $E$, and let $A_0, \ldots, A_n \in \calF$ be as in the definition of $T_2$-measurability. For each $i \in \{ 1, \ldots, n \}$, define $B_i := f(A_i)$, so that $B_i \times B_i \subseteq E$ and 
\[
\mu^* (\{ x \in X : f(x) \notin \cup_{i=1}^n B_i \}) \leq \mu(A_0) < \epsilon,
\]
which implies $f$ is smooth.

For~2, suppose $\calF$ is a $\sigma$-field and $f$ is conventionally measurable. Since the collection $\calO \subseteq \calU$ of open, symmetric entourages forms an entourage base (see Theorem~6, Chapter~6 of~\cite{kelley1963}), it is immediate that $f$ is uniformly base measurable. For 2b, consider $E \in \calU$. There exists $D \in \calO$ with $D \subseteq E$. Since $Y$ is Lindel\"of and $\{ D[y] : y \in Y \}$ is an open cover of $Y$, there is a sequence $\{ y_i \}_{i=1}^{\infty} \subseteq Y$ such that $Y = \cup_{i=1}^{\infty} D[y_i]$. Define $B_1 := D[y_1]$ and iteratively define $B_i := D[y_i] \setminus \cup_{j=1}^{i-1} B_j$ for $i \geq 2$. Then $A_i := f^{-1}(B_i) \in \calF$ and $f(A_i) \times f(A_i) \subseteq E$ for all $i \in \N$. Moreover, if $\mu$ is a bounded measure then for any $\epsilon > 0$ there is $n \in \N$ such that $\mu((\cup_{i = 1}^{n} A_i)^c) < \epsilon$. Define $A_0 := (\cup_{i = 1}^{n} A_i)^c \in \calF$, then $A_0, A_1, \ldots, A_n$ is a partition of the form required by the definition of $T_2$-measurability. 

For~3, suppose $(X, \calF, \mu)$ is a complete measure space and $Y$ is hereditarily Lindel\"of. Further suppose $f$ is base measurable, and consider an open set $G \subseteq Y$. For each $y \in G$, let $\calN_y$ be the (non-empty) set of all neighbourhoods $N$ of $y$ with $N \subseteq G$ and $f^{-1}(N) \in \overline{\calF}$. Let $\calN := \cup_{y \in G} \calN_y$. The interiors $\{ N^{\circ} : N \in \calN \}$ form an open cover of $G$, so there is a countable set $\{ N_i \}_{i = 1}^{\infty} \subseteq \calN$ such that $G = \cup_{i = 1}^{\infty} N_i$, hence $f^{-1}(G) = \cup_{i = 1}^{\infty} f^{-1}(N_i) \in \overline{\calF}$. But $\overline{\calF} = \calF$ by completeness of $(X, \calF, \mu)$, and thus $f$ is conventionally measurable. 

For~3b, consider an open set $G \subseteq Y$. For each $y \in G$, let $\calN_y$ be the set comprised of all pairs $(U, V)$, where $U$ is an open neighbourhood of $y$ contained in $G$ and $V \in \calF$ with $f^{-1}(U) \subseteq V \subseteq f^{-1}(G)$. If $f$ is $T_2$-measurable, then $\calN_y$ is non-empty, as evidenced by the following construction. There exists $D \in \calO$ with $(D \circ D)[y] \subseteq G$. Then $U := D[y]$ is an open neighbourhood of $y$ contained in $G$. Using the $T_2$-measurability of $f$, one may construct a pairwise disjoint sequence $\{ A_i \}_{i=1}^{\infty} \subseteq \calF$ such that $f(A_i) \times f(A_i) \subseteq D$ for all $i \in \N$ and $\mu^*(A_0) = 0$, where $A_0 := (\cup_{i=1}^{\infty} A_i)^c$. Since $(X, \calF, \mu)$ is complete, subsets of $A_0$ are elements of $\calF$. Now define 
\[
V := \bigcup \{ A_i : A_i \cap f^{-1}(U) \neq \emptyset \} \cup (A_0 \cap f^{-1}(U)).
\]
Then $V \in \calF$ with $f^{-1}(U) \subseteq V$. Moreover, $f(V) \subseteq (D \circ D)[y] \subseteq G$, hence $V \subseteq f^{-1}(G)$, as required. Now let $\calN := \cup_{y \in Y} \calN_y$. Since $Y$ is hereditarily Lindel\"of, there is a sequence $\{ (U_i, V_i) \}_{i = 1}^{\infty} \subseteq \calN$ such that $G = \cup_{i=1}^{\infty} U_i$ and thus $f^{-1}(G) = \cup_{i=1}^{\infty} f^{-1}(U_i) \subseteq \cup_{i=1}^{\infty} V_i \subseteq f^{-1}(G)$. It follows that $f^{-1}(G) = \cup_{i=1}^{\infty} V_i \in \calF$, that is, $f$ is conventionally measurable. 

To show~4, note that if $f$ is uniformly base measurable with entourage base $\calB$ as in Definition~\ref{measurable_uniform}(5), then $\calB_y := \{ E[y] : E \in \calB \}$ is a neighbourhood base at $y$ for any $y \in Y$, and $f^{-1}(B) \in \overline{\calF}$ for all $B \in \calB_y$. Hence $f$ is base measurable. For the partial converse, consider $E \in \calU$. If $Y$ is uniformly locally compact, there exists $D \in \calU$ such that $D \circ D \subseteq E$ and $D[y]$ is compact for all $y \in Y$. If $f$ is base measurable, for each $z \in D[y]$ the set $\calN_z$ comprised of all neighbourhoods $N$ of $z$ with $N \subseteq E[y]$ and $f^{-1}(N) \in \overline{\calF}$ is non-empty. Define $\calN := \cup_{z \in D[y]} \calN_z$. The interiors $\{ N^{\circ} : N \in \calN \}$ form an open cover of $D[y]$, hence there is a finite set $\{ N_i \}_{i = 1}^n \subseteq \calN$ that covers $D[y]$. Define $V_y := \cup_{i = 1}^n N_i$, then $D[y] \subseteq V_y \subseteq E[y]$ and $f^{-1}(V_y) \in \overline{\calF}$. But then $V := \cup_{y \in Y} \{ y \} \times V_y$ is an entourage with $D \subseteq V \subseteq E$ and $f^{-1}(V[y]) \in \overline{\calF}$ for all $y \in Y$, implying $f$ is uniformly base measurable.

To show~5, choose any $\epsilon > 0$ and any entourage $E$. Then there is some symmetric entourage $D$ with $D \circ D \subseteq E$. If $f$ is uniformly base measurable, there is some entourage $C \subseteq D$ such that $f^{-1}(C[y]) \in \overline{\calF}$ for all $y \in Y$. If $f$ is smooth, there are sets $G_1, \ldots, G_n \in \calP(Y)$ such that $G_i \times G_i \subseteq C$ for $i \in \{ 1, \ldots, n \}$ and $\mu^*(f^{-1}((\cup_{i=1}^n G_i)^c)) < \epsilon$. For each $i \in \{ 1, \ldots, n \}$, choose $y_i \in G_i$ and note $G_i \subseteq C[y_i]$ and $C[y_i] \times C[y_i] \subseteq D[y_i] \times D[y_i] \subseteq D \circ D \subseteq E$. Form a partition of $Y$ as follows: define $B_1 := C[y_1]$ and $B_i := C[y_i] \setminus \cup_{j=1}^{i-1} B_j$ for $i \in \{ 2, \ldots, n \}$. Then define $B_0 := (\cup_{i=1}^n B_i)^c \subseteq (\cup_{i=1}^n G_i)^c$. Set $A_i := f^{-1}(B_i)$ for $i \in \{ 0, \ldots, n \}$; then $A_0, A_1, \ldots, A_n$ is a partition of the form required by the definition of $T_2$-measurability. 
\end{proof}

The construction of the pairwise disjoint sequence $\{ A_i \}_{i=1}^{\infty}$ in the proof of~3b implicitly invokes the axiom of countable choice: one must choose a finite partition of $X$ (using $T_2$-measurability of $f$) for each element of a sequence $\epsilon_n \rightarrow 0$. The proof of~4 invokes the axiom of choice via selection of a finite cover for each $y \in Y$. Combining~4 and~5 gives that if $Y$ is uniformly locally compact and $f$ is base measurable and smooth, then $f$ is $T_2$-measurable. The latter result can alternatively be proved without invoking any choice axiom (proof omitted).

The second theorem in this section identifies a partial converse to Theorem~\ref{uniform_equivalences}(5), when $\calU$ is the weak (ie. initial) uniformity generated by a family of functions $\calS$. The conditions of the theorem require the following definitions and notations. 

\begin{definition} \label{boundary_mass_limit_defn}
Consider a charge space $(X, \calF, \mu)$ and a function $f : X \rightarrow \mathbb{R}$. For each $z \in \mathbb{R}$, define the {\em boundary mass limit} $\phi_f : \mathbb{R} \rightarrow [0, \infty]$ to be the function:
\[
\phi_f(z) := \lim_{\delta \rightarrow 0} \mu_*(f^{-1}(z - \delta, z + \delta)).
\]
\end{definition}

Let $\Xi$ denote the usual uniformity on $\mathbb{R}$. Given a uniform space $(Y, \calU)$, a function $g : Y \rightarrow \mathbb{R}$, and a set $V \subseteq \mathbb{R} \times \mathbb{R}$, define  
\[
g^{-1}(V) := \{ (y, y') \in Y \times Y : (g(y), g(y')) \in V \}
\]
and note $g^{-1}(V) \in \calU$ whenever $g$ is uniformly continuous and $V \in \Xi$. 

\begin{definition} \label{eta_and_zeta}
Consider a charge space $(X, \calF, \mu)$, a uniform space $(Y, \calU)$, and a function $f : X \rightarrow Y$. Let $\calS$ be a family of uniformly continuous real-valued functions on $Y$. For each $g \in \calS$, define
\begin{eqnarray*}
\epsilon_g &:=& \{ (-\infty, z), (z, \infty) : z \in \mathbb{R}, \phi_{g \circ f}(z) < \infty \}, \\
\eta_g &:=& \{ g^{-1}(E) : E \in \epsilon_g \}, \\
\calE_g &:=& \{ g^{-1}(E) : E \in \Xi, E[z] \in \epsilon_g \mbox{ for all } z \in \R \}, \\
\kappa_g &:=& \{ (-\infty, z), (z, \infty) : z \in \mathbb{R}, \phi_{g \circ f}(z) = 0 \}, \\
\zeta_g &:=& \{ g^{-1}(E) : E \in \kappa_g \}, \\
\calZ_g &:=& \{ g^{-1}(E) : E \in \Xi, E[z] \in \kappa_g \mbox{ for all } z \in \R \}.
\end{eqnarray*}
Also define $\eta := \cup_{g \in \calS} \eta_g$, $\calE := \cup_{g \in \calS} \calE_g$, $\zeta := \cup_{g \in \calS} \zeta_g$ and $\calZ := \cup_{g \in \calS} \calZ_g$. 
\end{definition}

Thus $\eta$ is the collection of all inverse images of open rays with finite boundary mass limit at the endpoint, and $\calE$ is a collection of entourages formed from elements of $\eta$. Similarly, $\zeta$ is the collection of all inverse images of open rays with zero boundary mass limit at the endpoint, and $\calZ$ is a collection of entourages formed from elements of $\zeta$.

The uniform continuity of the functions in $\calS$ ensures $\calE \subseteq \calU$ and $\calZ \subseteq \calU$. It also ensures the elements of $\eta$ and $\zeta$ are open sets (although continuous functions would be sufficient for this).

\begin{theorem} \label{T2_relationship_generalised}
Consider a charge space $(X,\calF,\mu)$, a uniform space $(Y, \calU)$, and a $T_2$-measurable function $f : X \rightarrow Y$. Let $\calS$, $\eta$ and $\calE$ be as in Definition~\ref{eta_and_zeta}.
\begin{enumerate}
\item If $\eta$ is a subbase for the uniform topology, then $f$ is base measurable.
\item If $\calE$ is an entourage subbase for $\calU$, then $f$ is uniformly base measurable.
\end{enumerate}
\end{theorem}

The proof depends on two lemmas concerning the properties of boundary mass limits, which will be stated and proved later in this section.

\begin{proof}
For~1, it will suffice to show that for any $E \in \eta$ and $y \in E$, there is $D \in \zeta$ with $y \in D \subseteq E$, for then $\zeta$ is also a subbase for the uniform topology. Moreover, $f^{-1}(D) \in \overline{\calF}$ for all $D \in \zeta$ by Lemma~\ref{phi_zero_general} below, hence $f$ is base measurable. So suppose $y \in E \in \eta$. Then $E = g^{-1}(V)$ for some $g \in \calS$ and $V \in \epsilon_g$, where $V$ is an interval of the form $(-\infty, v)$ or $(v, \infty)$ for some $v \in \R$ with $\phi_{g \circ f}(v) < \infty$. Note $g(y) \in V$. Lemma~\ref{phi_properties}(4) below implies there is $w$ strictly between $v$ and $g(y)$ with $\phi_{g \circ f}(w) = 0$. Hence $g(y) \in W \subseteq V$, where $W$ is whichever of the intervals $(-\infty, w)$ or $(w, \infty)$ contains $g(y)$. Then $D := g^{-1}(W) \in \zeta$ with $y \in D \subseteq E$, as required.

For~2, it will suffice to show that for any $E \in \calE$, there is $D \in \calZ$ with $D \subseteq E$, for then $\calZ$ is also an entourage subbase for $\calU$. Moreover, $f^{-1}(D[y]) \in \overline{\calF}$ for all $D \in \calZ$ and $y \in Y$ by Lemma~\ref{phi_zero_general} below, hence $f$ is uniformly base measurable. So suppose $E \in \calE$. Then $E = g^{-1}(V)$ for some $g \in \calS$ and $V \in \Xi$, where for each $z \in \R$, $V[z] \in \epsilon_g$ is an interval of the form $(-\infty, v_z)$ or $(v_z, \infty)$ for some $v_z \in \R$ with $\phi_{g \circ f}(v_z) < \infty$. Since $V \in \Xi$, there exists some $\gamma > 0$ such that $V \supset U_{\gamma} := \{ (u, u') \in \R \times \R : \lvert u - u' \rvert \leq \gamma \}$. Hence for each $z \in \R$, $[z - \gamma, z + \gamma] \subset V[z]$. Lemma~\ref{phi_properties}(4) implies there is $w_z$ strictly between $v_z$ and the interval $[z - \gamma, z + \gamma]$ with $\phi_{g \circ f}(w_z) = 0$. Hence $z \in W_z \subseteq V[z]$, where $W_z$ is whichever of the intervals $(-\infty, w_z)$ or $(w_z, \infty)$ contains $z$. Form the set $W := \cup_{z \in \R} \{ z \} \times W_z$. Then $W \in \Xi$, since $U_{\gamma} \subset W$, hence $D := g^{-1}(W) \in \calZ$.  Moreover, $W \subseteq V$ and thus $D \subseteq E$, as required.
\end{proof}

The proof of Part~2 of the preceding theorem implicitly invokes the axiom of choice, since $w_z$ must be chosen for each $z \in \R$. If $Y$ is uniformly locally compact, then~2 alternatively follows from~1 and from Theorem~\ref{uniform_equivalences}(4).

When $\calU$ is the weak uniformity generated by $\calS$ (for Part~2), or when the uniform topology is the weak topology generated by $\calS$ (for Part~1), it may be straightforward to establish the conditions of Theorem~\ref{T2_relationship_generalised}. For example, if $\mu$ is bounded then $\calE$ is trivially an entourage subbase for $\calU$ and $\eta$ is a subbase for the uniform topology, since $\phi_{g \circ f}^{-1}(\infty)$ is empty for all $g \in \calS$. Alternatively, if for each $A \in \calF$ either $\mu(A) < \infty$ or $\mu(A^c) < \infty$, then $\phi_{g \circ f}^{-1}(\infty)$ contains at most one element. 
If $Y$ is a Banach space and $f$ is integrable in the sense of Definition~III.2.17 of~\cite{dunford1958}, then $\phi_{g \circ f}^{-1}(\infty)$ does not contain any non-zero elements, and thus again the conditions of the theorem are met.

The main reason boundary mass limits are of interest here is that points of $\R$ with zero boundary mass limit induce sets in $Y$ with inverse images in $\overline{\calF}$, in the manner described in the following lemma. Note this result generalises Lemma~3.1 of~\cite{keith2022}, which applied only to real-valued, $T_2$-measurable functions. 

\begin{lemma} \label{phi_zero_general}
Consider a charge space $(X,\calF,\mu)$, a uniform space $Y$, and a function $f : X \rightarrow Y$. Suppose $g : Y \rightarrow \mathbb{R}$ is a uniformly continuous function with $\phi_{g \circ f}(z) = 0$ for some $z \in \mathbb{R}$. Define $D := g^{-1}(-\infty, z)$, and suppose at least one of the following conditions holds:
\begin{enumerate}
\item $f$ is $T_2$-measurable, or
\item $f$ is uniformly base measurable and $D$ is totally bounded.
\end{enumerate}
Then $f^{-1}(D) \in \overline{\calF}$ and $f^{-1}(\partial D) \in \overline{\calF}$ with $\overline{\mu}(f^{-1}(\partial D)) = 0$.
\end{lemma}

\begin{proof}
Define $h := g \circ f$. Since $\phi_h(z) = 0$, for any $\epsilon > 0$ there exists $\delta > 0$ such that $\mu_*(h^{-1}(z - \delta, z + \delta)) < \epsilon / 2$. 

Suppose $f$ is $T_2$-measurable. Then $h$ is $T_2$-measurable (this follows from the definitions of $T_2$-measurability and uniform continuity). Thus there is a partition of $X$ into sets $A_0, \ldots, A_n \in \calF$ such that $\mu(A_0) < \epsilon/2$ and $\lvert h(x) - h(x') \rvert < \delta$ for all $x, x' \in A_i$ and all $i \in \{ 1, \ldots, n \}$. Define $B$ to be the union of those elements of $\{ A_1, \ldots, A_n \}$ that are contained in $h^{-1}(-\infty, z)$, or the empty set if there are no such elements. Define $E$ to be the union of those elements of $\{ A_1, \ldots, A_n \}$ that intersect both $h^{-1}(-\infty, z]$ and $h^{-1}[z, \infty)$, or the empty set if there are no such elements. Then $E \subseteq h^{-1}(z - \delta, z + \delta)$, because if $A_i \subseteq E$ then for every pair $x, x' \in A_i$ with $h(x) \leq z \leq h(x')$ one must have $\lvert h(x) - h(x') \rvert < \delta$ and thus $\lvert h(x) - z \rvert < \delta$ and $\lvert h(x') - z \rvert < \delta$. Define $C := B \cup A_0 \cup E$. Then 
\[
\mu(C \setminus B) = \mu(A_0 \cup E) \leq \mu(A_0) + \mu_*(h^{-1}(z - \delta, z + \delta)) < \epsilon.
\]
Now $B, C \in \calF$ with $B \subseteq h^{-1}(-\infty, z) \subseteq C$, hence $f^{-1}(D) = h^{-1}(-\infty, z) \in \overline{\calF}$. Moreover, $f^{-1}(\partial D) \subseteq h^{-1}(z) \subseteq C \setminus B$, since $\partial D \subseteq g^{-1}(z)$. Hence $f^{-1}(\partial D) \in \overline{\calF}$ with $\overline{\mu}(f^{-1}(\partial D)) = 0$.

Alternatively, suppose $f$ is uniformly base measurable and $D$ is totally bounded. 
Since $g$ is uniformly continuous, the set $\{ (y, y') : \lvert g(y) - g(y') \rvert < \delta \}$ is an entourage, and thus contains an entourage $E$ with $f^{-1}(E[y]) \in \overline{\calF}$ for all $y \in Y$. Total boundedness of $D$ then implies there is some finite set $F \subseteq Y$ such that $\overline{D} \subseteq \cup_{y \in F} E[y]$. Define $B := \bigcup \{ f^{-1}(E[y]) : y \in F, E[y] \subseteq D \}$, or the empty set if there are no elements in that union. Define $C := \cup_{y \in F} f^{-1}(E[y])$. Then $B, C \in \overline{\calF}$ with $B \subseteq f^{-1}(D) \subseteq C$ and $C \setminus B \subseteq f^{-1}(g^{-1}(z - \delta, z + \delta))$. The latter implies $\overline{\mu}(C \setminus B) \leq \epsilon$, hence $f^{-1}(D) \in \overline{\calF}$. Moreover, $f^{-1}(\partial D) \subseteq C \setminus B$, hence $f^{-1}(\partial D) \in \overline{\calF}$ with $\overline{\mu}(f^{-1}(\partial D)) = 0$.
\end{proof}


In order to use Lemma~\ref{phi_zero_general} to show that a function is base measurable or uniformly base measurable, the proof of Theorem~\ref{T2_relationship_generalised} required the set of points in $\R$ with zero boundary mass limit to be dense in the set of points with finite boundary mass limit. Surprisingly, this turns out to be true for any real-valued function on a charge space, even if the function is not measurable in any sense. This is a consequence of the following lemma, which generalises Lemma~3.2 of~\cite{keith2022}: the latter lemma applied only to $T_2$-measurable functions on a bounded charge space.  

\begin{lemma} \label{phi_properties}
Consider a charge space $(X, \calF, \mu)$ and a function $f : X \rightarrow \mathbb{R}$. The boundary mass limit $\phi_f$ has the following properties.
\begin{enumerate}
\item For any $z \in \mathbb{R}$ with $\phi_f(z) = 0$, and any $\epsilon > 0$, there is $\delta > 0$ with $\phi_f(w) < \epsilon$ for all $w \in (z - \delta, z + \delta)$.

\item For any $z \in \phi_f^{-1}(0, \infty)$, there is $\delta > 0$ such that $\phi_f(z) > \sum_{i = 1}^n \phi_f(w_i)$ for any $w_1, \ldots, w_n \in (z - \delta, z + \delta) \setminus \{ z \}$.

\item The set $\phi_f^{-1}(0, \infty)$ is countable.

\item $\phi_f^{-1}[0, \infty)$ is an open set in which $\phi_f^{-1}(0)$ is dense.
\end{enumerate}
\end{lemma}

\begin{proof}
Property~1 can be shown by contradiction. Suppose it is false, so that there is $z \in \mathbb{R}$ with $\phi_f(z) = 0$ and $\epsilon > 0$ such that for any $\delta > 0$ there is $w \in (z - \delta, z + \delta)$ with $\phi_f(w) \geq \epsilon$. Consider any such $\delta$ and $w$,  then $\mu_*(f^{-1}(z - \delta, z + \delta)) \geq \phi_f(w) \geq \epsilon$. Letting $\delta \rightarrow 0$ gives $\phi_f(z) \geq \epsilon$, a contradiction.

Property~2 can similarly be shown by contradiction. Suppose it is false, implying there is $z \in \phi_f^{-1}(0, \infty)$ such that for any $\delta > 0$ there are $w_1, \ldots, w_n \in (z - \delta, z + \delta) \setminus \{ z \}$ with $\phi_f(z) \leq \sum_{i=1}^n \phi_f(w_i)$. Choose $\delta^{(1)} > 0$ and obtain $w_1^{(1)}, \ldots, w_{n_1}^{(1)}  \in (z - \delta^{(1)}, z + \delta^{(1)}) \setminus \{ z \}$ with $\phi_f(z) \leq \sum_{i=1}^n \phi_f(w_i^{(1)})$. Choose $\delta^{(2)} \in (0, \delta^{(1)})$ such that the intervals $\{ (w_i^{(1)} - \delta^{(2)}, w_i^{(1)} + \delta^{(2)}) \}_{i=1}^{n_1}$ and the interval $(z - \delta^{(2)}, z + \delta^{(2)})$ are pairwise disjoint subintervals of $(z - \delta^{(1)}, z + \delta^{(1)})$. But then there are $w_1^{(2)}, \ldots, w_{n_2}^{(2)} \in (z - \delta^{(2)}, z + \delta^{(2)}) \setminus \{ z \}$ with $\phi_f(z) \leq \sum_{i=1}^{n_2} \phi_f(w_i^{(2)})$. This process can be iterated for all $k \in \N$ to obtain $\delta^{(k)}$ and $w_1^{(k)}, \ldots, w_{n_k}^{(k)} \in (z - \delta^{(k)}, z + \delta^{(k)}) \setminus \{ z \}$. Now $(z - \delta^{(1)}, z + \delta^{(1)})$ contains all the pairwise disjoint intervals 
\[
\{ (w_i^{(k)} - \delta^{(k+1)}, w_i^{(k)} + \delta^{(k+1)}) : i \in \{ 1, \ldots, n_k\}, k \in \N \}.
\]
Hence for any $K \in \N$,
\begin{eqnarray*}
\mu_*(f^{-1}(z - \delta^{(1)}, z + \delta^{(1)})) & \geq & \sum_{k=1}^K \sum_{i=1}^{n_k} \mu_*(f^{-1}(w_i^{(k)} - \delta^{(k+1)}, w_i^{(k)} + \delta^{(k+1)})) \\
& \geq & \sum_{k=1}^K \sum_{i=1}^{n_k} \phi_f(w_i^{(k)}) \\
& \geq & K \phi_f(z).
\end{eqnarray*}
Thus $\mu_*(f^{-1}(z - \delta^{(1)}, z + \delta^{(1)})) = \infty$, since $\phi_f(z) > 0$. But this holds for any $\delta^{(1)} > 0$, hence $\phi_f(z) = \infty$, a contradiction.

To show Property~3, consider $z \in \phi_f^{-1}(0, \infty)$ and $\delta_z > 0$ with the property specified in~2. Then $\phi_f^{-1}(1/n, \infty) \cap (z - \delta_z, z + \delta_z)$ must be a finite set for each $n \in \N$, hence $\phi_f^{-1}(0, \infty) \cap (z - \delta_z, z + \delta_z)$ is countable. The sets 
\[
\{ \phi_f^{-1}(0, \infty) \cap (z - \delta_z, z + \delta_z) : z \in \phi_f^{-1}(0, \infty) \}
\]
form an open cover of $\phi_f^{-1}(0, \infty)$ in the subspace topology. Hence there is a countable subcover, since $\mathbb{R}$ is hereditarily Lindel\"of. Thus $\phi_f^{-1}(0, \infty)$ is a countable union of countable sets, and is therefore countable.

For~4, note~1 and~2 together imply $\phi_f^{-1}[0, \infty)$ is an open set. Consider $z \in \phi_f^{-1}(0, \infty)$ and any $\delta > 0$ with the property specified in~2. Then $(z - \delta, z + \delta)$ does not contain any element of $\phi_f^{-1}(\infty)$. Moreover, $\phi_f^{-1}(0, \infty) \cap (z - \delta, z + \delta)$ is countable, as in the proof of~3.  Hence $(z - \delta, z + \delta)$ must contain an element of $\phi_f^{-1}(0)$, implying $\phi_f^{-1}(0)$ is dense in $\phi_f^{-1}[0, \infty)$.
\end{proof}

In the proof of Part~3 of the preceding lemma, the claim that $\mathbb{R}$ is hereditarily Lindel\"of is equivalent to the claim that the axiom of countable choice holds for subsets of $\R$~\cite{herrlich1997}. Note Part~3 is not used in any other proofs.

\section{Measurable real-valued functions} \label{real_valued_functions}

The various notions of measurability described in the previous two sections will now be applied to real-valued functions. Moreover, additional notions of measurability specific to real-valued functions will be introduced.

A real-valued function $f$ on a measurable space $(X, \calF)$ is conventionally measurable iff $f^{-1}(y, \infty) \in \calF$ for all $y \in \mathbb{R}$. (Some authors define measurability of real-valued functions in this way: see Definition 121C and Theorem~121E(f) of~\cite{fremlin2011}.) This condition can be generalised several ways.

\begin{definition} \label{ray_measurable}
Consider a charge space $(X, \calF, \mu)$. A function $f : X \rightarrow \mathbb{R}$ is 
\begin{enumerate}
\item {\em right ray measurable} if $f^{-1}(y, \infty) \in \overline{\calF}$ for all $y \in D$, where $D$ is a dense subset of $\mathbb{R}$; 

\item {\em left ray measurable} if $f^{-1}(-\infty, y) \in \overline{\calF}$ for all $y \in D$, where $D$ is a dense subset of $\mathbb{R}$; 

\item {\em ray measurable} if both $f^{-1}(y, \infty) \in \overline{\calF}$ and $f^{-1}(-\infty, y) \in \overline{\calF}$ for all $y \in D$, where $D$ is a dense subset of $\mathbb{R}$;



\item {\em regularly $T_1$-measurable} if there is a sequence $\{ s_i \}_{i=1}^{\infty}$ of $\overline{\calF}$-simple functions converging hazily to $f$,  where
\begin{enumerate}
\item $s_i := \sum_{k=1}^{n_i} y_{ik} (I_{A_{ik}^+} - I_{A_{ik}^-})$ with $n_i := i 2^i - 1$; 
\item $y_{ik} := k 2^{-i} \delta$ for each $k \in \{ 1, \ldots, n_i + 1 \}$, $i \in \N$ and some $\delta > 0$; 
\item $A_{ik}^+ := f^{-1}( y_{ik}, y_{i,k+1} ] \in \overline{\calF}$ and $A_{ik}^- := f^{-1}[-y_{i,k+1}, -y_{ik}) \in \overline{\calF}$ for each $k \in \{ 1, \ldots, n_i \}$ and $i \in \N$.
\end{enumerate}
\end{enumerate}
\end{definition}

Some logical relationships between the different forms of measurability can be stated immediately. In what follows, the term ``$\overline{\calF}$-measurable'' is used in the sense of Definition~\ref{greco_measurable}. 

\begin{proposition} \label{ray_and_base_proposition}
Consider a charge space $(X, \calF, \mu)$ and a function $f : X \rightarrow \mathbb{R}$.
\begin{enumerate}
\item If $f$ is ray measurable then it is left ray measurable and right ray measurable.


\item If $f$ is left or right ray measurable then it is $\overline{\calF}$-measurable.

\item If $f$ is $\overline{\calF}$-measurable then it is base measurable.

\item $f$ is uniformly base measurable iff it is base measurable.

\item $f$ is ray measurable iff $f^+$ and $f^-$ are ray measurable. 



\item If $f$ is ray measurable then $\psi(f)$ is ray measurable for any homeomorphism $\psi : \R \rightarrow \R$. (Similar statements hold for $\calF$-measurable and base measurable functions.)

\item If $\calF$ is a $\sigma$-field and $f$ is conventionally measurable, then $f$ is ray measurable. 

\item If $(X, \calF, \mu)$ is a complete measure space and $f$ is ray measurable, then $f$ is conventionally measurable. 

\end{enumerate}
\end{proposition}

Claim~4 is immediate from Theorem~\ref{uniform_equivalences}(4), since $\R$ is uniformly locally compact. Proofs of the other claims are straightforward, and are omitted.  

Additional relationships between these various types of measurability hold under the conditions of the following theorem, which generalises Theorem~3.4 of~\cite{keith2022} in two ways: 1) it allows $\mu$ to be an unbounded charge; and 2) it pertains to additional forms of measurability.

\begin{theorem} \label{real_valued_measurability}
Consider a charge space $(X, \calF, \mu)$ and a function $f : X \rightarrow \mathbb{R}$. Suppose $\phi_f^{-1}[0, \infty)$ is dense in $\mathbb{R}$. 

\begin{enumerate}
\item The following statements are logically equivalent.
\begin{enumerate}
\item $f$ is ray measurable.
\item $f$ is left ray measurable.
\item $f$ is right ray measurable.
\item $f$ is $\overline{\calF}$-measurable.
\item $f^+$ and $f^-$ are base measurable.
\item $f$ is base measurable and $\exists y \in \mathbb{R}$ with $f^{-1}(y, \infty) \in \overline{\calF}$.
\end{enumerate}

\item The following statements are logically equivalent.
\begin{enumerate}
\item $f$ is regularly $T_1$-measurable.
\item $f$ is $T_1$-measurable.
\item $f$ is $T_2$-measurable.
\item $f$ is ray measurable and smooth.
\item $f$ is base measurable and smooth. 
\end{enumerate}

\end{enumerate}
\end{theorem}

\begin{proof}
The implications 1a $\implies$ 1b $\implies$ 1d $\implies$ 1e follow from Proposition~\ref{ray_and_base_proposition}(1, 2, and 3) and the fact that $f$ is $\overline{\calF}$-measurable only if $f^+$ and $f^-$ are $\overline{\calF}$-measurable. The implications 1a $\implies$ 1c $\implies$ 1d follow similarly.

(1e $\implies$ 1a) Suppose $f^+$ and $f^-$ are base measurable. Then $f^+$ is uniformly base measurable by Proposition~\ref{ray_and_base_proposition}(4). Define $Y := [0, \infty)$ with the subspace uniformity and let $g : Y \rightarrow \R$ denote the inclusion map. Note $g$ is uniformly continuous and $g^{-1}(-\infty, z) = [0,z)$ is totally bounded for any $z \in \R$. It follows that $(f^+)^{-1}(-\infty, z) \in \overline{\calF}$ and $(f^+)^{-1}(z, \infty) \in \overline{\calF}$ for any $z \in \R$ with $\phi_{f^+}(z) = 0$, by the second condition of Lemma~\ref{phi_zero_general}. But $\phi_{f^+}^{-1}(0)$ is dense in $\R$ because $\phi_{f}^{-1}(0)$ is dense in $\R$ by Lemma~\ref{phi_corollary} below, and $\phi_{f}^{-1}(0) \setminus \{ 0 \} \subseteq \phi_{f^+}^{-1}(0)$. Hence $f^+$ is ray measurable, and a similar argument gives $f^-$ is ray measurable. Thus $f$ is ray measurable by Proposition~\ref{ray_and_base_proposition}(5).

(1a $\iff$ 1f) Suppose $f$ is ray measurable. Then $f$ is base measurable by Proposition~\ref{ray_and_base_proposition}(1, 2, and 3), and the extra condition of 1f holds trivially. Conversely, suppose 1f holds. Then $(f - y)^+$ and $(f - y)^-$ are base measurable, noting $((f - y)^+)^{-1}(B) = f^{-1}(-\infty, y] \cup f^{-1}(y+B)$ and $((f - y)^-)^{-1}(B) = f^{-1}(y, \infty) \cup f^{-1}(y-B)$, for any neighbourhood $B$ of $0$. Hence $f - y$ is ray measurable, since 1e $\implies$ 1a. But then $f$ is ray measurable by Proposition~\ref{ray_and_base_proposition}(6).

(2a $\implies$ 2b $\implies$ 2c) If $f$ is regularly $T_1$-measurable, then it is $T_1$-measurable with respect to the charge space $(X, \overline{\calF}, \overline{\mu})$, hence also $T_1$-measurable with respect to $(X, \calF, \mu)$ by Proposition~1.8(c) of~\cite{basile2000}. The second implication follows by Theorem~4.4.7 of~\cite{bhaskararao1983}.

(2c $\implies$ 2a) By Lemma~\ref{phi_corollary} below, the set $\phi_f^{-1}(0)$ is comeagre in $\mathbb{R}$. But then the set 
\[
\{ \delta \in (0, \infty) : \phi_f(k 2^{-i} \delta) > 0 \mbox{ for some } i \in \N, k \in \mathbb{Z} \}
\]
is meagre in $(0, \infty)$, because it is a countable union of meagre sets. Since $(0, \infty)$ is not meagre in itself, there is some $\delta \in (0, \infty)$ such that $\phi_f(k 2^{-i} \delta) = 0$ for all $i \in \N$ and $k \in \mathbb{Z}$. Now, if $f$ is $T_2$-measurable, then $f^{-1}(-\infty, k 2^{-i} \delta) \in \overline{\calF}$ and $f^{-1}(k 2^{-i} \delta, \infty) \in \overline{\calF}$ for all $i \in \N$ and $k \in \mathbb{Z}$, by the first condition of Lemma~\ref{phi_zero_general} with $g$ the identity function on $Y = \R$. One can therefore define a sequence $\{ s_i \}_{i=1}^{\infty}$ of $\overline{\calF}$-simple functions as in Definition~\ref{ray_measurable}(4). Since $f$ is smooth, then for any $\epsilon, M > 0$ there exists some $k \in \N$ such that $\mu^*(\{ x \in X : \lvert f(x) \rvert > k  \}) < \epsilon$ and $2^{-k} < M$. But then for $i \geq k$, 
\[
\mu^*(\{ x \in X : \lvert s_i(x) - f(x) \rvert > M \}) \leq \mu^*(\{ x \in X : \lvert f(x) \rvert > k  \}) < \epsilon.
\]
Thus $s_i$ converges hazily to $f$, implying $f$ is regularly $T_1$-measurable.

(2c $\implies$ 2d) If $f$ is $T_2$-measurable, then $f$ is smooth. Moreover, $f^{-1}(-\infty, z) \in \overline{\calF}$ and $f^{-1}(z, \infty) \in \overline{\calF}$ for any $z \in \R$ with $\phi_{f}(z) = 0$, by the first condition of Lemma~\ref{phi_zero_general} with $g$ the identity function on $Y = \R$. Hence $f$ is ray measurable, since $\phi_{f}^{-1}(0)$ is dense in $\R$ by Lemma~\ref{phi_corollary}.

(2d $\implies$ 2e $\implies$ 2c) The first implication holds by Proposition~\ref{ray_and_base_proposition}(1, 2, and 3). The second implication holds by Theorem~\ref{uniform_equivalences}(4 and 5).
\end{proof}

The requirement that $\phi_f^{-1}[0, \infty)$ be dense in $\mathbb{R}$ could be replaced by any of the equivalent conditions mentioned in Lemma~\ref{phi_corollary} below. Note this requirement is only invoked in the proofs of 1e $\implies$ 1a (which is itself invoked in the proof of 1f $\implies$ 1a), 2c $\implies$ 2a and 2c $\implies$ 2d.
This requirement is a special case of the condition of Theorem~\ref{T2_relationship_generalised}(1) - that $\eta$ is a topological subbase. To see this, suppose that in Definition~\ref{eta_and_zeta}, $Y = \mathbb{R}$ and $\calS$ contains only the identity function $\iota(y) = y$. Then $\eta = \{ (-\infty, z), (z, \infty) : z \in \phi_f^{-1}[0, \infty) \}$, which is a sub-base for the usual topology on $\mathbb{R}$ if $\phi_f^{-1}[0, \infty)$ is dense in $\mathbb{R}$. Consequently, the implication 2c $\implies$ 2e alternatively follows from Theorem~\ref{T2_relationship_generalised}(1). 

The claim that the set defined in the proof of 2c $\implies$ 2a above is meagre because it is a countable union of meagre sets can be made here without invoking the axiom of countable choice, because each of the meagre sets in the union can be expressed as a countable union of nowhere dense sets with a canonical ordering (see proof of 2~$\implies$~3 in Lemma~\ref{phi_corollary} below).

As noted in the introduction, a smooth, $\overline{\calF}$-measurable function on a charge space must be $T_2$-measurable, a consequence of Theorem~3 of~\cite{greco1981}. Theorem~\ref{real_valued_measurability} above implies the converse holds if $\phi_f^{-1}[0, \infty)$ is dense in $\mathbb{R}$.



The following lemma, used in the proof of Theorem~\ref{real_valued_measurability}, may be regarded as a corollary to Lemma~\ref{phi_properties}, highlighting how points with zero, finite, and infinite boundary mass limit are distributed in $\R$. 

\begin{lemma} \label{phi_corollary}
Let $I$ be any open interval (bounded or unbounded) in $\mathbb{R}$. The following statements are logically equivalent.
\begin{enumerate}
\item $\phi_f^{-1}(\infty) \cap I$ contains no open intervals.
\item $\phi_f^{-1}[a, \infty] \cap I$ is nowhere dense in $I$ for all $a \in (0, \infty]$.
\item $\phi_f^{-1}(0) \cap I$ is comeagre in $I$.
\item $\phi_f^{-1}(0) \cap I$ is dense in $I$.
\item $\phi_f^{-1}[0, \infty) \cap I$ is dense in $I$.
\end{enumerate}
\end{lemma}

\begin{proof}
1~$\implies$~2 because if there exists $a \in (0, \infty]$ such that $\phi_f^{-1}[a, \infty] \cap I$ is dense in some open subinterval $J$ of $I$, Lemma~\ref{phi_properties}(1 and 2) imply $J$ contains no element of $\phi_f^{-1}[0, \infty)$, hence $J \subseteq \phi_f^{-1}(\infty) \cap I$. 2~$\implies$~3 because $\phi_f^{-1}(0, \infty] \cap I = \cup_{n \in \N} \phi_f^{-1}[1/n, \infty] \cap I$. 3~$\implies$~4 because $I$ is a Baire space, hence a comeagre set is dense. 4~$\implies$~5 because $\phi_f^{-1}(0) \subseteq \phi_f^{-1}[0, \infty)$. 5~$\implies$~1 because a dense subset of an interval intersects every open subinterval.
\end{proof}

\section{Inheritance of measurability}

This final section identifies conditions under which measurability of a function $f : X \rightarrow Y$ is ``inherited'' by a coarser or finer uniformity. The three main results are stated in the following theorem.

\begin{theorem} \label{PJ_transfer}
Consider a charge space $(X, \calF, \mu)$ and a function $f : X \rightarrow Y$.
\begin{enumerate}
\item If $f$ is $T_2$-measurable or smooth with respect to a uniformity $\calU_1$ on $Y$, then $f$ is respectively $T_2$-measurable or smooth with respect to any coarser uniformity $\calU_2 \subseteq \calU_1$.

\item If $f$ is base measurable wrt each topology on $Y$ in a collection $\calC$, then $f$ is base measurable wrt $\bigvee \calC$.

\item If $f$ is uniformly base measurable wrt each uniformity on $Y$ in a collection $\calC$, then $f$ is uniformly base measurable wrt $\bigvee \calC$.
\end{enumerate}
\end{theorem}

In Claim~2 of this theorem, the {\em least upper bound} $\bigvee \calC$ of a collection of topologies $\calC$ is the smallest topology that contains $\bigcup \calC$. Similarly, in Claim~3, the {\em least upper bound} $\bigvee \calC$ of a collection of uniformities $\calC$ is the smallest uniformity that contains $\bigcup \calC$.

Claim~1 is immediate from the definitions of $T_2$-measurability and smoothness. Claim~2 follows from the fact that a base for $\bigvee \calC$ can be obtained by taking finite intersections of elements drawn from a union of bases for each of the topologies in $\calC$. Statement~3 follows in a similar manner. Thus none of the claims requires detailed proof. Nevertheless, the theorem has some useful corollaries.

\begin{corollary} \label{PJ_pseudometrics}
Consider a charge space $(X,\calF,\mu)$, a uniform space $Y$ with uniformity $\calU$ induced by a family of pseudometrics $\calS$, and a function $f : X \rightarrow Y$. 
\begin{enumerate}
\item If $f$ is $T_2$-measurable with respect to $\calU$, then $f$ is $T_2$-measurable with respect to the uniformities induced by each $p \in \calS$.

\item If $f$ is uniformly base measurable wrt the uniformities induced by each $p \in \calS$, then $f$ is uniformly base measurable wrt $\calU$.
\end{enumerate}
\end{corollary}

\begin{corollary} \label{PJ_coarser_uniformity}
Consider a charge space $(X,\calF,\mu)$, a uniform space $(Y, \calU)$ with the weak uniformity $\calU$ induced by a family of functions $\{ h_{\alpha} : \alpha \in \Omega \}$, where $h_{\alpha} : Y \rightarrow Y_{\alpha}$ and $(Y_{\alpha}, \calU_{\alpha})$ is a uniform space for each $\alpha \in \Omega$, and a function $f : X \rightarrow Y$. 

\begin{enumerate}
\item If $f$ is $T_2$-measurable with respect to $\calU$, then $h_{\alpha} \circ f$ is $T_2$-measurable with respect to $\calU_{\alpha}$ for each $\alpha \in \Omega$.

\item If $h_{\alpha} \circ f$ is uniformly base measurable with respect to $\calU_{\alpha}$ for each $\alpha \in \Omega$, then $f$ is uniformly base measurable wrt $\calU$.
\end{enumerate}
\end{corollary}

In particular, Corollary~\ref{PJ_coarser_uniformity} applies when $Y := \prod_{\alpha \in \Omega} Y_{\alpha}$ with the product uniformity, and $h_{\alpha}$ is the coordinate projection of $Y$ onto $Y_{\alpha}$ for each $\alpha \in \Omega$.


\begin{thebibliography}{00}


\bibitem{weil1937} A.~Weil, Sur les espaces a structure uniforme et sur la topologie g\'en\'erale, Actualit\'es Scientifique et Industrielles, 551, Publications de l'Institut Math\'ematique de l'Universit\'e de Strasbourg, 1937.

\bibitem{kelley1963} J.~L.~Kelley, General Topology, D. Van Nostrand Company, New York, 1963.

\bibitem{bourbaki1966} N.~Bourbaki, General Topology Part 1, Hermann, Paris, 1966.

\bibitem{james1987} I.~M.~James, Topological and Uniform Spaces, Undergraduate Texts in Mathematics, Springer-Verlag, New York, 1987.

\bibitem{james1990} I.~M.~James, Introduction to Uniform Spaces, London Mathematical Society lecture note series 144, Cambridge University Press, New York, 1990.

\bibitem{encyclopedia2003} K.~P.~Hart, J.-I.~Nagata, J.~E.~Vaughan, J.~E.~Vaughan, Encyclopedia of General Topology, Elsevier Science, Oxford, 2003.

\bibitem{willard2004} S.~Willard, General Topology, Dover Publications, New York, 2004.

\bibitem{avallone1991} A.~Avallone and A.~Basile, Integration: Uniform Structure, Journal of Mathematical Analysis and Applications, 159 (1991) 373-381.

\bibitem{dunford1958} N.~Dunford and J.~T.~Schwartz, Linear Operators Part I: General Theory, Pure and Applied Mathematics 7, Wiley-Interscience, New York, 1958.

\bibitem{dudley2017} R.~M.~Dudley, Real Analysis and Probability, CRC Press, 2017.

\bibitem{athreya2006} K.~B.~Athreya and S.~N.~Lahiri, Measure Theory and Probability Theory, Springer, New York, 2006.

\bibitem{keith2022} J.~M.~Keith, Properties of functions on a bounded charge space, Analysis and Geometry in Metric Spaces, 10 (2022) 63-89.

\bibitem{basile2000} A.~Basile and K.~P.~S.~Bhaskara Rao, Completeness of $L_p$-Spaces in the Finitely Additive Setting and Related Stories, Journal of Mathematical Analysis and Applications, 248 (2000) 588-624.

\bibitem{bhaskararao1983} K.~P.~S.~Bhaskara Rao and M.~Bhaskara Rao, Theory of Charges, Pure and Applied Mathematics 109, Academic Press, New York, 1983.

\bibitem{greco1981} G.~H.~Greco, Sur la mesurabilit\'e d'une fonction num\'erique par rapport \`a une famille d'ensembles, Rendiconti del Seminario Matematico della Universit\`a di Padova, 65 (1981) 163--176. (Translation: arXiv:2309.04049)

\bibitem{herrlich1997} H.~Herrlich and G.~E.~Strecker, When is $\N$ Lindel\"of?, Commentationes Mathematicae Universitatis Carolinae, 38(3) (1997) 553--556.

\bibitem{fremlin2011} D.~H.~Fremlin, Measure Theory Volume 1, Torres Fremlin, 2011.


\end{thebibliography}
\end{document}